\documentclass[11pt,reqno]{amsart}
\usepackage{amssymb,latexsym}

\usepackage{graphicx}
\usepackage{fancyhdr}
\usepackage[all]{xy}
\usepackage{enumerate}
\theoremstyle{plain}
\numberwithin{equation}{section}
\newtheorem{thm}{Theorem}[section]
\newtheorem{theorem}[thm]{Theorem}

\newtheorem{corollary}[thm]{Corollary}

\newtheorem{lemma}[thm]{Lemma}

\newtheorem{definition}[thm]{Definition}

\newtheorem*{theorem*}{Theorem}

\begin{document}

\setcounter{page}{1}

\title[Multipliers of nilpotent Lie superalgebras]{Multipliers of nilpotent Lie superalgebras}
\author[Nayak]{Saudamini Nayak}
\address{Harish-Chandra Research Institute, HBNI, \\
         Chhatnag Road, Jhunsi, 
          Allahabad-211 019 \\
                India}
\email{ anumama.nayak07@gmail.com}

\thanks{Research supported by the Harish-Chandra Research Institute Postdoctoral Fellowship.}

\begin{abstract}\label{abstract}
In this paper, first we prove that all finite dimensional special Heisenberg Lie superalgebras with even center have same dimension, say  $(2m+1\mid n)$ for some non-negative integers $m,n$ and are isomorphism with them. Further, for a nilpotent Lie superalgebra $L$ of dimension $(m\mid n)$ and $\dim (L') = (r\mid s)$ with $r+s \geq 1$, we find the upper bound $\dim \mathcal{M}(L)\leq \frac{1}{2}\left[(m + n + r + s - 2)(m + n - r -s -1) \right] + n + 1$, where $\mathcal{M}(L)$ denotes the Schur multiplier of $L$. Moreover, if $(r, s) =(1, 0)\; (\mathrm{respectively}\; (r,s) = (0,1))$, then the equality holds if and only if $L \cong H(1,0) \oplus A_{1}\; (\mathrm{respectively}\; H(0,1) \oplus A_{2})$, where $A_{1}$ and $A_{2}$ are abelian Lie superalgebras with $\dim A_{1}=(m-3 \mid n), \dim A_{2}=(m-1 \mid n-1)$ and $H(1,0), H(0,1)$ are special Heisenberg Lie superalgebras of dimension $3$ and $2$ respectively.
\end{abstract}

\maketitle

\section{Introduction}\label{intro}
For a given group $G$, the notion of the Schur multiplier $\mathcal{M}(G)$ arose from the work of Schur on projective representation of groups as the second cohomology group with coefficients in $\mathbb{C}^{*}$. Let $0\longrightarrow R \longrightarrow F \longrightarrow G \longrightarrow 0$ be a free presentation of the group $G$, then it can be shown by Hopf's formula that $\mathcal{M}(G) \cong R\cap [F,F]/[R,F]$. Green in \cite{Green1956} proved that for every finite $p$-group $G$ of order $p^{n}$, $\mathcal{M}(G) = p^{\frac{1}{2}n(n-1)-t(G)}$ for some non-negative integer $t(G)$. Using the notion $t(G)$, one can easily classify finite $p$-groups. For example, Berkovich \cite{Berkovich1991} showed that $t(G) = 0$ if and only if $G$ is an elementary abelian $p$-group. Also, he classified all $p$-groups $G$ with $t(G)=1$. Similarly, the case  $t(G) = 2$ was studied by Zhou in  \cite{Zhou1994} and determined the structure of $p$-groups.

\smallskip
The finite dimensional Lie algebra analogue to the Schur multiplier was developed in \cite{Batten1993} and later it has been studied by several authors. Let $A$ be a Lie algebra over a field $\mathbb{F}$ with a free presentation $0\longrightarrow R \longrightarrow F \longrightarrow A \longrightarrow 0$, where $F$ is a free Lie algebra. Then the Schur multiplier of $A$, denoted by $\mathcal{M}(A)$, is defined to be the factor Lie algebra $R \cap[F,F]/[R,F]$ \cite{Batten1993}. It is easy to see that the Schur multiplier of a Lie algebra $A$ is abelian and independent of the choice of the free presentation of $A$. For  more information about the Schur multiplier of Lie algebras one can refer \cite{Batten1993,BMS1996,BS1996,NR2011} and the references given therein. Batten in \cite{Batten1993} showed that if $A$ is finite dimensional, then its Schur multiplier is isomorphic to $H^{2}(A,\mathbb{F})$, the second cohomology group of $A$.  Following the result of Green, Moneyhun in \cite{Moneyhun1994}, showed that for a Lie algebra $A$ of dimension $n$, $\dim \mathcal{M}(A) = \frac{1}{2}n(n-1)-t(A)$, where $t(A)$ is non-negative integer. For finite dimensional nilpotent Lie algebras, Batten et al., in \cite{BMS1996} obtained the result similar to Berkovich and Zhou. 

\smallskip
 
Recently, Niroomand in \cite{Nir2009} improved the result of Jones \cite[Theorem 3.1.4]{Kar1987} on the Schur multiplier of a non-abelian $p$-group. He proved that a non-abelian $p$-group $G$ of order $p^{n}$ with derived subgroup of order $p^{k}$ has, $\mathcal{M}(G)\leq p^{\frac{1}{2}(n+k-2)(n-k-1)+1}$. Further, Niroomand and Russo in \cite{NR2011}, have extended this result to the Lie algbera case. In particular, they have proved that $\dim \mathcal{M}(A)\leq \frac{1}{2}(n+m-2)(n-m-1)+1$, for a Lie algebra $A$ of dimension $n$ and for  dimension of its derived algebra $A'$ equal to $m$. In \cite{ESD2016}, it has been proved that if $A$ is nilpotent $n$-Lie algebra of dimension $d$ with $\dim A'= m \geq 1$, then $\dim \mathcal{M}(A)\leq \binom{d-m+1}{n} +(m-2)\binom{d-m}{n-1}+n-m$. In this paper, we study similar problem for Lie superalgebras.

\smallskip
The organization of the paper is as follows. Section \ref{pre} and \ref{bounds}, give some preliminary ideas about Lie superalgebras and some useful inequalities for $\mathcal{M}(L)$. In Section \ref{shl}, we study special Heisenberg Lie superalgebras and finally in Section \ref{mt}, we obtain a bound for the multiplier of finite dimensional nilpotent Lie superalgebras. 

\section{Preliminaries}\label{pre}
In this section we review some terminology on Lie superalgebras (see \cite{KAC1977, Musson2012}) and recall notions used in the paper. Throughout, $\mathbb{F}$ is an algebraically closed field of characteristic zero and all vector spaces and algebras are over $\mathbb{F}$ and are of finite dimensions. A $\mathbb{Z}_2$-graded vector space is also called a superspace. Let $V = V_{\bar{0}} \oplus V_{\bar{1}}$ be a superspace. We call elements of $V_{\bar{0}}$ (resp. $V_{\bar{1}}$) even (resp. odd). Non-zero elements of $V_{\bar{0}} \cup V_{\bar{1}}$ are {\it homogeneous} and for homogeneous $v \in V_{\sigma}$ with $\sigma \in \mathbb{Z}_2$, we set $|v| = \sigma$ the degree of $v$.  
\subsection{Definition of Lie superalgebra}
A {\it Lie superalgebra} is a superspace $L = L_{\bar{0}} \oplus L_{\bar{1}}$ with a bilinear mapping
$ [., .] : L \times L \rightarrow L$ satisfying the following identities:

\begin{enumerate}
\item $[L_{\alpha}, L_{\beta}] \subset L_{\alpha+\beta}$, for $\alpha, \beta \in \mathbb{Z}_{2}$ ($\mathbb{Z}_{2}$-grading),
\item $[x, y] = -(-1)^{|x||y|} [y, x]$ (graded skew-symmetry),
\item $(-1)^{|x||z|} [x,[y, z]] + (-1)^{ |y| |x|} [y, [z, x]] + (-1)^{|z| |y|} [z,[ x, y]] = 0$ (graded Jacobi identity),
\end{enumerate}
for all $x,y,z \in L$. 

\smallskip
For a Lie superalgebra $L = L_{\bar{0}} \oplus L_{\bar{1}}$, the even part $L_{\bar{0}}$ is a Lie algbera and $L_{\bar{1}}$ is $L_{\bar{0}}$-module. A Lie superalgebra without even part, i.e., $L_{\bar{0}} = 0$, is an abelian Lie superalgebra, that is $[x,y] = 0$ for all $x,y \in L$. Let $I$ and $J$ be two sub(super)algebras of a  Lie superalgebra $L$. We denote by $[I,J]$ the sub(super)algebra of $L$ spanned by all elements $[m,n]$ with $m \in I$ and $n\in J$. A $\mathbb{Z}_{2}$-graded subalgebra $I$ is a {\it graded ideal} of $L$ if $[I,L]\subseteq I$. In particular, the subalgebra $Z(L) = \{z\in L : [z,x] = 0\;\mathrm{for all}\;x\in L\}$ is a graded ideal and it is called the {\it center} of $L$. Clearly, if $I$ and $J$ are graded ideals of $L$, then so is $[I,J]$. For a graded ideal $I$ of $L$, the quotient $L/I$ inherits a canonical Lie superalgebra structure such that the natural projection map becomes a homomorphism, as defined below.

\smallskip
A linear mapping $f: V \longrightarrow W$ is said to be {\it homogeneous} of degree $\gamma,\; \gamma \in \mathbb{Z}_{2}$, if 
\[f(V_{\alpha}) \subset W_{\alpha + \gamma}\quad \mbox{for all}\;\; \alpha \in \mathbb{Z}_{2}.\]
The mapping $f$ is called a {\it homomorphism} of the $\mathbb{Z}_{2}$ vector space $V$ in to the $\mathbb{Z}_{2}$ vector space $W$ if $f$ is homogeneous of degree $0$. A Lie superalgebra homomorphism $f: V \rightarrow W$ is a  homomorphism (homogeneous linear mappings of degree $0$) such that $f[x,y] = [f(x), f(y)]$ for all $x,y \in V$. The notions of {\it epimorphisms, isomorphisms} and {\it auotomorphisms} have the obvious meaning. We write the superdimension of Lie superalgebra $L$ as $(m\mid n)$, where $\dim L_{\bar{0}} = m$ and $\dim L_{\bar{1}} = n$. 

\subsection{Free Lie superalgebra}
\begin{definition} [See p. 135 in \cite{Musson2012}]
The free Lie superalgebra on a $\mathbb{Z}_{2}$-graded set $X = X_{\bar{0}} \cup X_{\bar{1}}$ is a Lie superalgebra $F(X)$ together with a degree zero map $i: X \rightarrow F(X)$ such that if $M$ is any Lie superalgebra and $j: X \rightarrow M$ is a degree zero map, then there is a unique Lie superalgebra homomorphism $h: F(X) \rightarrow M$ with $j = h \circ i$.
\end{definition}
The existence of free Lie superalgebra is guaranteed by an analogue of Witt's theorem (see \cite[Theorem 6.2.1]{Musson2012}).
 
\smallskip
If $L$ is a Lie superalgebra generated by a $\mathbb{Z}_{2}$-graded set $X = X_{\bar{0}} \cup X_{\bar{1}}$  and $\phi : X \rightarrow L$ is a degree zero map, then there exists a free Lie superalgebra $F$ and  $\psi: F \rightarrow L$ extending $\phi$. Let $R = \ker (\psi)$. The extension 
\begin{equation}\label{eq0}
0 \longrightarrow R \longrightarrow F \longrightarrow L \longrightarrow 0
\end{equation} 
is called a {\it free presentation} of $L$ and is denoted by $(F, \psi)$. With this free presentation of $L$,  we define {\it multiplier} of $L$ as  
$$ 
\mathcal{M}(L) = \frac{[F,F]\cap R}{[F, R]}.
$$
Now from the following lemma one can see that multiplier of Lie superalgebra $L$ is isomorphic to $H_2(L)$, the second homology of $L$.
\begin{lemma}[ See Corollary 6.5(iii) in \cite{GKL2015}]
 Suppose equation \eqref{eq0}  is a free presentation of a Lie
superalgebra $L$. Then there is an isomorphism of supermodules
\[H_2(L) = \frac{[F,F]\cap R}{[F, R]}.\]
\end{lemma}

\subsection{Extensions of Lie superalgebras}

An extension of a Lie superalgebra $L$ is a short exact sequence 
\begin{equation}\label{eq1}
\xymatrix{
0 \ar[r] & M\ar[r]^{e} & K\ar[r]^{f} & L \ar[r] & 0}. 
\end{equation}
Since $e: M \longrightarrow e(M) = \ker(f)$ is an isomorphism we will usually identify $M$ and $e(M)$. An extension of $L$ is then same as an epimorphism $f: K \longrightarrow L$. A homomorphism from an extension $f: K \longrightarrow L$ to another extension $f': K' \longrightarrow L$ is a Lie superalgebra homomorphism $g: K \longrightarrow K'$ satisfying $f = f' \circ g$; in other words, we have a commutative diagram, 
\begin{center}
$\xymatrix{
K \ar[rd]_{f} \ar[rr]^{g} & & K' \ar[ld]^{f'} \\
& L &} $ 
\end{center}
The concept of stem Lie algebra is defined and studied in \cite{Batten1993}. Here we define stem extension and stem cover analogously for Lie superalgebras. 

A {\it central extension} of $L$ is an extension \eqref{eq1} such that $M \subset Z(K)$. The central extension is said to be a {\it stem extension} of $L$ if $M \subseteq Z(K)\cap K' $. The stem extension is {\it maximal} if every epimorphism of any other stem extension of $L$ on to $0 \longrightarrow M \longrightarrow  K \longrightarrow   L \longrightarrow 0$ is necessarily an  isomorphism. Finally, we call the stem extension a {\it stem cover} if $M \cong \mathcal{M}(L)$ and in this case $K$ is said to be a cover of Lie superalgebra $L$. Equivalently, we define the stem cover as maximal defining pair for a Lie superalgebra $L$ as follows. 
\begin{definition}\label{def3}
A pair of Lie superalgebras $(K,M)$ is said to be a defining pair for $L$ if 
\begin{enumerate}
\item $L \cong K/M$,
\item $M\subset Z(K)\cap K'$.
\end{enumerate}
\end{definition}
In this paper, we have alternatively used the defination of maximal defining pair and stem cover for Lie superalgebra $L$.

\subsection{Nilpotent Lie superalgebras}
For an arbitrary Lie superlagebra one can define a descending central sequence as follows :
\begin{align*}
C^{0}(L) &= L \\
C^{k}(L) &= [ L, C^{k-1}(L)]\;\;\mbox{for}\;\; k\geq 1.
\end{align*}
\begin{definition}\label{def1}
A Lie superalgebra is called nilpotent if there exists a
positive integer $n$ such that $C^{n}(L) = 0$.
\end{definition}
 We define two sequences for a Lie superalgebra $L = L_{\bar{0}} \oplus L_{\bar{1}}$ as follows:
\begin{align*}
C^{0}(L_{\bar{0}}) &= L_{\bar{0}}, \quad C^{k} (L_{\bar{0}}) = [L_{\bar{0}}, C^{k-1} (L_{\bar{0}})]\;\;\mbox{for}\;\; k\geq 1  \\
C^{0}(L_{\bar{1}}) &= L_{\bar{1}}, \quad C^{k} (L_{\bar{1}}) = [L_{\bar{0}}, C^{k-1} (L_{\bar{1}})] = \mathrm{ad}_{L_{\bar{0}}}(C^{k-1} (L_{\bar{1}})\;\;\mbox{for}\;\; k\geq 1.
\end{align*}
\begin{lemma}[\cite{MG2001}]\label{lem2}
The Lie superalgebra $L = L_{\bar{0}} \oplus L_{\bar{1}}$ is nilpotent if and only if there exist  $(p, q)\in \mathbb{N} \times \mathbb{N}$ such that
$$
C^{p-1}(L_{\bar{0}}) \neq 0,\;\; C^{q-1}(L_{\bar{1}}) \neq 0\; \quad \mathrm{and}\quad C^{p}(L_{\bar{0}}) = C^{q}(L_{\bar{1}}) = {0}.$$
\end{lemma}

\section{Bounds for $\mathcal{M}(L)$}\label{bounds}

In this section, we prove some inequalities and bounds for dimension of $\mathcal{M}(L)$ which are important ingredients to prove the main theorem in Section \ref{mt}. Schur proved that if $G$ is a group such that the order of $G/Z(G)$ is finite, then so is the derived subgroup of $G$. Similarly, Moneyhun \cite{Moneyhun1994} proved for a Lie algbera that if $\dim A/Z(A) = n$ then $\dim A' \leq n(n-1)/2$. The following theorem is a analogue result in Lie superalgebra.

\begin{theorem}\label{lem4}
Let $L$ be a Lie superalgebra with 
$\dim \left( L/Z(L)\right) = (m \mid n)$. Then 
$$\dim L'\leq \frac{1}{2}\left[(m+n)^2 + (n-m)\right].$$
\end{theorem}
\begin{proof}
Let $\{\bar{v}_{1},\cdots, \bar{v}_{m};\bar{v}_{m+1},\cdots, \bar{v}_{m+n}\}$ be a basis for $L/Z(L)$. Then the generating sets for $L'$ are $\{[v_{i},v_{j}]\mid 1\leq i< j\leq m\},\; \{[v_{k},v_{l}] \mid m+1\leq k \leq l\leq m+n\}$ and $\{[v_{i},v_{j}]\mid 1\leq i\leq m\; \mathrm{and}\;m + 1\leq j\leq m + n\}$. Using the skew-super symmetric property of Lie superalgebra, we have
\begin{align*}
\dim L'\leq & \binom{m}{2} + \binom{n}{2} + n + mn\\
		  = & \frac{1}{2}\left[(m+n)^2 + (n-m)\right].
\end{align*}
\end{proof}

\begin{corollary}\label{lem5}
Let $L$ be a Lie superalgebra with 
$\dim\ L= (m\mid n)$ and $0\longrightarrow M \longrightarrow K \longrightarrow L \longrightarrow 0$ be a stem extension of $L$. Then $\dim K \leq \frac{1}{2}\left[(m+n)^2 + (m + 3n)\right]$. 
\end{corollary}

\begin{proof}
Since both $K$ and $M$ are Lie superalgebras, 
$$
(K/M)' = ((K_{\bar{0}} \oplus K_{\bar{1}})/(M_{\bar{0}} \oplus M_{\bar{1}}))' = K'_{\bar{0}} + K'_{\bar{1}} + [K_{\bar{0}}, K_{\bar{1}}] + (M_{\bar{0}} \oplus M_{\bar{1}}) = K' + M,$$ 
where even and odd part of $K'$ are $K'_{\bar{0}} + K'_{\bar{1}}$ and $[K_{\bar{0}}, K_{\bar{1}}]$ respectively. Using the definition of stem extension, we have 
$$
\dim (K/Z(K))\leq \dim (K/M) = \dim\ L = (m\mid n).
$$

By Theorem \ref{lem4}, $\dim M\leq \dim K'\leq \frac{1}{2}\left[(m+n)^2 + (n-m)\right]$.
Thus, 
\begin{align*}
\dim K = & m + n + \dim M  \\
	   \leq & m + n + \frac{1}{2}\left[(m+n)^2 + (n-m)\right]\\
	   = & \frac{1}{2}\left[(m+n)^2 + (m + 3 n)\right].
\end{align*}
\end{proof}
 
From Corollary \ref{lem5}, one can conclude that for a finite dimensional Lie superalgebra, $\dim K$ is bounded. Thus, we can find $K$ of maximal dimension.

\begin{definition}\label{def6}
If $\dim K$ is maximal then $(K,M)$ is called a maximal defining pair. For this maximal defining pair, $K$ is called a cover for $L$ and $M$ is called a multiplier which is denoted as $\mathcal{M}(L)$.
\end{definition} 

\begin{theorem}\label{lem4a}
Let $L$ be a Lie superalgebra with 
$\dim \L = (m \mid n)$. Then 
$$\dim \mathcal{M}(L)\leq \frac{1}{2}\left[(m+n)^2 + (n-m)\right].$$
\end{theorem}
\begin{proof}
Let $0\longrightarrow M \longrightarrow K \longrightarrow L \longrightarrow 0$ be a stem cover of $L$. As $M \subseteq Z(K)$, we have $\dim (K/Z(K))\leq \dim (K/M) = \dim\ L = (m\mid n)$. Hence by Theorem $\ref{lem4}$, $\dim K' \leq \frac{1}{2}\left[(m+n)^2 + (n-m)\right]$. Again since $M \subseteq K'$ and $M \cong \mathcal{M}(L)$, it follows that 
\[\dim \mathcal{M}(L) = \dim M \leq \dim K' \leq \frac{1}{2}\left[(m+n)^2 + (n-m)\right],\]
as required.
\end{proof}

\begin{theorem}\label{lem8}
Let $L$ be a Lie superalgebra with $\dim\ L = (m\mid n)$. Then  
$$
\dim \mathcal{M}(L) = \frac{1}{2}\left[(m+n)^2 + (n-m)\right]$$ 
if and only if $L$ is abelian.
\end{theorem}

\begin{proof}
Suppose $L$ is an abelian Lie superalgebra of dimension $(m\mid n)$ with basis \\
$\{v_{1}, v_2,\cdots, v_m; v_{m+1},\cdots, v_{m+n}\}$. Let $(K,W)$ be a defining pair for $L$. Let $\pi: K \longrightarrow L$ be an onto homomorphism such that $W=\ker \pi$. Consider a section $\mu$ of homomorphism $\pi$, which  is a homomorphism from $L\longrightarrow K$ such that $\pi \circ \mu =Id$. For $1\leq i\leq m+n$ define $\mu(v_{i}) = x_{i}$ and $y_{i,j}= [x_{i}, x_{j}]$. Here $y_{i,j}\in W$ since \begin{equation*}
\pi(y_{i,j}) = \pi [x_{i}, x_{j}] = \pi [\mu(v_i), \mu(v_j)] = [\pi \circ \mu(v_i),\pi \circ \mu(v_j)] = [v_i,v_j] = 0.
\end{equation*}
Now the generating set for $W$ are
$\{y_{i,j} = [x_{i},x_{j}]\mid 1\leq i< j\leq m\},\; \{y_{k,l}= [x_{k},x_{l}] \mid m+1\leq k \leq l\leq m+n\}$ and $\{y_{r,s} =  [x_{r},x_{s}]\mid 1\leq r\leq m\; \mathrm{and}\;m+1\leq s\leq m+n\}$. Let $J = \mu(L)$ and $K$ is a vector space direct sum of $J$ and $W$. Therefore, $x_{i}\; \left(1\leq i \leq m+n\right)$ and $y_{i,j}\; (1\leq i< j\leq m), y_{k,l}\; (m+1\leq k \leq l\leq m+n), y_{r,s} \left(1\leq r\leq m, m+1\leq s\leq m+n\right)$ generate $K$. Since $L$ is abelian and $W\subset Z(K)$, bracket of any three elements in $K$ is zero. Hence $(K,W)$ is maximal defining pair for $L$. Thus $W$ is the multiplier of $L$ and 
$$
\dim \mathcal{M}(L) = \dim W = \frac{1}{2}\left[(m+n)^2 + (n-m)\right].$$

Conversely, assume that $\dim \mathcal{M}(L) = \frac{1}{2}\left[(m+n)^2 + (n-m)\right]$. Let $(K, M)$ be a maximal defining pair for $L$. From the proof of Corollary \ref{lem5}, we have $\dim M \leq \dim K'\leq \frac{1}{2}\left[(m+n)^2 + (n-m)\right]$. Thus we have, 
$$
\frac{1}{2}\left[(m+n)^2 + (n-m)\right] = \dim \mathcal{M}(L)  = \dim M \leq \dim K' \leq \frac{1}{2}\left[(m+n)^2 + (n-m)\right].$$ Therefore we conclude that $M = K'$ and hence $L \cong K/M = K/K'$ is abelian. 
\end{proof}

The following result is proved in \cite[Theorem 2.5.2]{Kar1987} for groups and in \cite[Lemma 4]{BMS1996} for Lie algebras. Following their ideas we prove the same for Lie superalgebras. Though the proofs are in the same line, we have provided it for the sake of completeness.

\begin{lemma}\label{lem9}
Let $L$ be a finite dimensional Lie superalgebra with graded ideal $K$ and set $H = H/K = (L/K)_{\bar{0}} \oplus (L/K)_{\bar{1}}$. Then there exists a finite dimensional Lie superalgebra $G=G_{\bar{0}} \oplus G_{\bar{1}} $ with a graded ideal $M$ such that 
\begin{enumerate}[(i)]
\item $L'\cap K \cong G/M$, i.e., $(L'_{\bar{0}} + L'_{\bar{1}})\cap K_{\bar{0}} \cong \left(G/M \right)_{\bar{0}}$ and   $[L_{\bar{0}} , L_{\bar{1}}]\cap K_{\bar{1}} \cong \left(G/M\right)_{\bar{1}},$
\item $M \cong \mathcal{M}(L),$
\item $\mathcal{M}(H)$ is a homomorphic image of $G$,
\item if $K\subset Z(L)$, then $L'\cap K$ is an epimorphic image of $\mathcal{M}(H).$
\end{enumerate}
\end{lemma}

\begin{proof}
Let $ 0 \longrightarrow R_{\bar{0}}\oplus R_{\bar{1}} \longrightarrow F_{\bar{0}}\oplus F_{\bar{1}} \longrightarrow L_{\bar{0}}\oplus L_{\bar{1}} \longrightarrow 0 $ be a free presentation of Lie superalgebra $L$ and $K= (T/ R)_{\bar{0}} \oplus (T/ R)_{\bar{1}}$ for some graded ideal $T$ of $F$. 

Now, \begin{align*}
L' \cap K & = \left(\frac{F_{\bar{0}} \oplus F_{\bar{1}}}{R_{\bar{0}} \oplus R_{\bar{1}}}\right)' \cap \left(\frac{T_{\bar{0}} \oplus T_{\bar{1}}}{R_{\bar{0}} \oplus R_{\bar{1}}} \right)\\
& =\left(\frac{\left(F_{\bar{0}} \oplus F_{\bar{1}}\right)' + R_{\bar{0}} \oplus R_{\bar{1}}}{R_{\bar{0}} \oplus R_{\bar{1}}} \right)\cap \left(\frac{T_{\bar{0}} \oplus T_{\bar{1}}}{R_{\bar{0}} \oplus R_{\bar{1}}}\right)\\
& = \left(\frac{F_{\bar{0}}' + F_{\bar{1}}' + R_{\bar{0}}}{R_{\bar{0}}} \oplus \frac{[F_{\bar{0}},  F_{\bar{1}}] + R_{\bar{1}}}{R_{\bar{1}}} \right) \cap \left( \frac{T_{\bar{0}}}{R_{\bar{0}}} \oplus\frac{T_{\bar{1}}}{R_{\bar{1}}} \right)\\
& = \frac{((F_{\bar{0}}' + F_{\bar{1}}') + R_{\bar{0}}) \cap T_{\bar{0}}}{R_{\bar{0}}} \oplus  \frac{([F_{\bar{0}}, F_{\bar{1}}] + R_{\bar{1}}) \cap T_{\bar{1}}}{R_{\bar{1}}} \cong \frac{(F_{\bar{0}}' + F_{\bar{1}}')\cap T_{\bar{0}}}{(F_{\bar{0}}' + F_{\bar{1}}')\cap R_{\bar{0}}} \oplus  \frac{[F_{\bar{0}}, F_{\bar{1}}]  \cap T_{\bar{1}}}{[F_{\bar{0}}, F_{\bar{1}}] \cap R_{\bar{1}}} \\
& = \frac{((F_{\bar{0}}' + F_{\bar{1}}')\cap T_{\bar{0}})/([F_{\bar{0}}, R_{\bar{0}}] + [F_{\bar{1}}, R_{\bar{1}}])}{((F_{\bar{0}}' + F_{\bar{1}}')\cap R_{\bar{0}})/([F_{\bar{0}}, R_{\bar{0}}] + [F_{\bar{1}}, R_{\bar{1}}])}\oplus \frac{([F_{\bar{0}},F_{\bar{1}}]\cap T_{\bar{1}})/([F_{\bar{0}},R_{\bar{1}}])}{([F_{\bar{0}},F_{\bar{1}}]\cap R_{\bar{1}})/([F_{\bar{0}},R_{\bar{1}}])}.  
\end{align*}
The above isomorphism comes from 3rd isomorphism theorem and the last equality is well defined as $[F_{\bar{0}}, R_{\bar{0}}] + [F_{\bar{1}}, R_{\bar{1}}] \subset F_{\bar{0}}' + F_{\bar{1}}',\; [F_{\bar{0}}, R_{\bar{1}}] \subset [F_{\bar{0}}, F_{\bar{1}}]$ and 
$[F_{\bar{0}}, R_{\bar{0}}] + [F_{\bar{1}}, R_{\bar{1}}] \subset R_{\bar{0}} + R_{\bar{0}} \subset T_{\bar{0}},\; [F_{\bar{0}}, R_{\bar{1}}] \subset  R_{\bar{1}} \subset T_{\bar{1}}$. Set
\begin{align*}
G_{\bar{0}} = ((F_{\bar{0}}' + F_{\bar{1}}')\cap T_{\bar{0}})/([F_{\bar{0}}, R_{\bar{0}}] + [F_{\bar{1}}, R_{\bar{1}}]), &\; G_{\bar{1}} = ([F_{\bar{0}},F_{\bar{1}}]\cap T_{\bar{1}})/([F_{\bar{0}},R_{\bar{1}}]),\\
M_{\bar{0}} = (F_{\bar{0}}' + F_{\bar{1}}')\cap R_{\bar{0}})/([F_{\bar{0}}, R_{\bar{0}}] + [F_{\bar{1}}, R_{\bar{1}}]),& \; M_{\bar{1}} = ([F_{\bar{0}},F_{\bar{1}}]\cap R_{\bar{1}})/([F_{\bar{0}},R_{\bar{1}}]).
\end{align*}
Then $L' \cap K \cong \left(G/M \right)_{\bar{0}} \oplus \left(G/M\right)_{\bar{1}}$, which completes the proof of (i). Clearly from  definition, $\mathcal{M}(L) = M_{\bar{0}} \oplus M_{\bar{1}} = M$. Thus (ii) hold. Also, 
$$H_{\bar{0}} \oplus H_{\bar{1}} = (L/K)_{\bar{0}} \oplus (L/K)_{\bar{1}} = \frac{(F/R)_{\bar{0}}}{(T/R)_{\bar{0}}} \oplus \frac{(F/R)_{\bar{1}}}{(T/R)_{\bar{1}}} = (F/T)_{\bar{0}} \oplus (F/T)_{\bar{1}} = F/T.$$ 
Thus, by definition 
$$
\mathcal{M}(H) = \frac{F'\cap T}{[F, T]} =\frac{(F' \cap T)/[F,R]}{[F,T]/[F,R]} = \frac{G}{([F,T]/[F,R])}.$$ Therefore, $\mathcal{M}(H)$ is the image of $G$ under some homomorphism, whose kernel is $[F,T]/[F,R]$, which completes part (iii). From the proof of part (i), we have 
\[L'\cap K \cong \frac{(F_{\bar{0}}' + F_{\bar{1}}')\cap T_{\bar{0}}}{(F_{\bar{0}}' + F_{\bar{1}}')\cap R_{\bar{0}}} \oplus \frac{[F_{\bar{0}}, F_{\bar{1}}] \cap T_{\bar{1}}}{[F_{\bar{0}}, F_{\bar{1}}] \cap R_{\bar{1}}}.\] Since $K\subset Z(L), \left(\frac{T}{R}\right)_{\bar{0}}\subset (Z\left(\frac{F}{R}\right))_{\bar{0}} \subset Z\left(\frac{F}{R}\right)_{\bar{0}}$, which implies that $[T_{\bar{0}}, F_{\bar{0}}]\subset R_{\bar{0}}$. Similarly, we have $[T_{\bar{1}}, F_{\bar{1}}]\subset R_{\bar{1}}$. Therefore, 
\begin{align*}
L'\cap K &\cong \frac{(F_{\bar{0}}' + F_{\bar{1}}')\cap T_{\bar{0}}}{(F_{\bar{0}}' + F_{\bar{1}}')\cap R_{\bar{0}}} \oplus \frac{[F_{\bar{0}}, F_{\bar{1}}] \cap T_{\bar{1}}}{[F_{\bar{0}}, F_{\bar{1}}] \cap R_{\bar{1}}} \\
& = \frac{((F_{\bar{0}}' + F_{\bar{1}}')\cap T_{\bar{0}})/ [T_{\bar{0}}, F_{\bar{0}}]}{((F_{\bar{0}}' + F_{\bar{1}}')\cap R_{\bar{0}})/[T_{\bar{0}}, F_{\bar{0}}]} \oplus \frac{([F_{\bar{0}}, F_{\bar{1}}] \cap T_{\bar{1}})/[T_{\bar{1}}, F_{\bar{1}}]}{([F_{\bar{0}}, F_{\bar{1}}] \cap R_{\bar{1}})/[T_{\bar{1}}, F_{\bar{1}}]} =
 \frac{\mathcal{M}(H)}{(F'\cap R)/[T,F]}.
\end{align*}
This completes the proof of lemma.
\end{proof}

\begin{corollary}\label{cor10}
Let $L$ be a finite dimensional Lie superalgebra and $K$ be a graded ideal of $L$ and $H=L/K$. Then, 
$
\dim \mathcal{M}(H)\leq \dim \mathcal{M}(L) + \dim (K\cap L').$
\end{corollary}
\begin{proof}
From Lemma \ref{lem9}, we have $\dim G = \dim (K\cap L') + \dim \mathcal{M}(L)$ and $\dim \mathcal{M}(H) \leq \dim G$. Thus, $\dim \mathcal{M}(H)\leq \dim (K\cap L') + \dim \mathcal{M}(L).$
\end{proof}
The following result is proved in \cite[Theorem 2.5.5(ii)]{Kar1987} for groups and for Lie algebra in \cite[Corollary 2.3]{NR2011}.

\begin{theorem}\label{th11}
Let $L$ be a finite dimensional Lie superalgebra, $K \subset Z(L)$ an graded ideal and $H = L/K$. Then
\begin{enumerate}
\item $\dim \mathcal{M}(L) + \dim (L'\cap K) \leq \dim \mathcal{M}(H) + \dim \mathcal{M}(K) + \dim (H/H'\otimes K/K');$
\item $ \dim \mathcal{M}(L) + \dim (L'\cap K)  \leq \frac{1}{2}[(m+n)^2 +(n-m)].$
\end{enumerate}

\end{theorem}
\begin{proof}
Consider a free presentation of $L: 0\longrightarrow R \longrightarrow F \longrightarrow L \longrightarrow 0$.  $K$ is a central ideal of $L$. Suppose $K =T/R$ for some graded ideal $T$ of $F$, then $H= L/K \cong F/T$. Now by proof of Lemma \ref{lem9}, 
\begin{equation}\label{eq11a}
\dim \mathcal{M}(L) + \dim L'\cap K = \dim G = \dim \mathcal{M}(H) + \dim ([F,T]/[F,R]).
\end{equation}
Since $
\frac{[F,T]/[F,R]}{([F,R]+T')/[F,R]}\cong [F,T]/([F,R]+T')$, we have
\begin{align*}
\dim \mathcal{M}(L) + \dim (L'\cap K) =  \dim \mathcal{M}(H) + \dim (([F,R]+T')/[F,R]) + \dim ([F,T]/([F,R]+T')).
\end{align*}
Also, $([F,R]+T')/[F,R] \cong T'/([F,R]\cap T')\cong \frac{T'/[T,R]}{([F,R]\cap T')/[T,R]}$. Since $T'\subset [F,T]\subset R$ we have $T'/[T,R] = (T' \cap R)/[T,R] \cong \mathcal{M}(K)$. Therefore,
\begin{align*}
\dim \mathcal{M}(L) + \dim (L'\cap K) =&  \dim \mathcal{M}(H) + \dim \mathcal{M}(K) - \dim (([F,R]\cap T')/[T,R])\\ & + \dim ([F,T]/([F,R]+T'))\\
& \leq \dim \mathcal{M}(H) + \dim \mathcal{M}(K) + \dim ([F,T]/([F,R]+T')).
\end{align*}
We define a map $f: \frac{F}{F'+ T}\times \frac{T}{R}\longrightarrow \frac{[F,T]}{[F,R]+ T'}$ and given by $f(\bar{a}, \bar{b}) = \overline{[a,b]}$. It is easy to see that $f$ is an epimorphism. Also, $\frac{F}{F'+ T}\times \frac{T}{R} \cong H/H' \times K/K'$. Thus,
$$
\dim \mathcal{M}(L) + \dim (L'\cap K) \leq \dim \mathcal{M}(H) + \dim \mathcal{M}(K) + \dim (H/H' \times K/K'),$$
and this completes the proof of part (1). Since $H\cong F/T$, we have $\mathcal{M}(H) = \frac{F'\cap T}{[F, T]} \cong \frac{(F'\cap T)/[F,R]}{[F,T]/[F,R]}$. Now from \eqref{eq11a}, we can conclude that 
$$
\dim \mathcal{M}(L) + \dim (L'\cap K)  \leq \dim \left(\frac{F'\cap T}{[R,F]}\right)\leq \dim \left(\frac{F}{[R,F]}\right)'.$$ 
Using Theorem \ref{lem4} and  
$$
\dim \left(\frac{F/[R,F]}{Z(F/[R,F])}\right) \leq  \dim \left(\frac{F/[R,F]}{R/[R,F]}\right) = \dim (F/R) = (m \mid n),$$ we infer that $\dim (F/[R,F])' \leq \frac{1}{2}[(m+n)^2 +(n-m)]$. This completes the proof of part (2).
\end{proof}
\begin{corollary}\label{cor11b}
Let $L$ be a Lie superalgebra of dimension $(m\mid n)$. Then 
\begin{equation*}
\dim \mathcal{M}(L) + \dim L' \leq \frac{1}{2}[(m+n)^2 +(n-m)]. 
\end{equation*}
\end{corollary}
\begin{proof}
The result follows by putting $K = L'$ in part (2) of Theorem \ref{th11}. 
\end{proof}

The Schur multiplier of the direct product of two finite groups is equal to the direct product of the Schur multiplier of the two factors plus the tensor product of the abelianization of two groups \cite[Theorem2.2.10]{Kar1987}. In \cite{BMS1996}, Batten et al., proved the same result for Lie algberas. Further using K$\ddot{\mathrm{u}}$nneth formula, Niroomand et al., in \cite{NR2011} gave a short proof of this result. Following the idea of \cite{BMS1996} we prove the same for Lie superalgebra.
\begin{theorem}\label{th12}
Let $A$ and $B$ finite dimensional Lie superalgebras. Then 
\begin{equation*}
\dim \mathcal{M}(A\oplus B) = \dim \mathcal{M}(A) + \dim \mathcal{M}(B) + \dim (A/A'\otimes B/B'). 
\end{equation*}
\end{theorem}
\begin{proof}
Consider a short exact sequence of $A\oplus B$ i.e.,
$$
0 \rightarrow N \rightarrow H \rightarrow A\oplus B \rightarrow 0 $$
where $A\oplus B \cong H/N$ and $N \subset Z(H)\cap H'$. Let the Lie superalgebra pair $(H, N)$ be the maximal defining pair for $A\oplus B$. This implies $\mathcal{M}(A\oplus B) = N$. Since $A = A_{\bar{0}} \oplus A_{\bar{1}}$ and $B = B_{\bar{0}} \oplus B_{\bar{1}}$ are Lie superalgebras so is $A\oplus B$ with even part $(A\oplus B)_{\bar{0}} = A_{\bar{0}} + B_{\bar{0}}$ and odd part $(A\oplus B)_{\bar{1}} = A_{\bar{1}} + B_{\bar{1}}$. There exist graded ideals $X= X_{\bar{0}} + X_{\bar{1}}, Y= Y_{\bar{0}} + Y_{\bar{1}}$ of $H = H_{\bar{0}} + H_{\bar{1}}$ such that 
$$
\frac{ X_{\bar{0}} \oplus X_{\bar{1}}}{N_{\bar{0}} \oplus N_{\bar{1}}} = \frac{X_{\bar{0}}}{N_{\bar{0}}} \oplus \frac{X_{\bar{1}}}{N_{\bar{1}}} \cong A_{\bar{0}} \oplus A_{\bar{1}}\;\;\mathrm{and}\;\; \frac{ Y_{\bar{0}} \oplus Y_{\bar{1}}}{N_{\bar{0}} \oplus N_{\bar{1}}}  \cong B_{\bar{0}} \oplus B_{\bar{1}}.$$
Now, 
$$\left(\frac{X_{\bar{0}}}{N_{\bar{0}}} + \frac{Y_{\bar{0}}}{N_{\bar{0}}}\right) \oplus \left(\frac{X_{\bar{1}}}{N_{\bar{1}}} + \frac{Y_{\bar{1}}}{N_{\bar{1}}}\right) \cong( A_{\bar{0}} + B_{\bar{0}}) \oplus (A_{\bar{1}} + B_{\bar{1}}) \cong \frac{H_{\bar{0}}}{N_{\bar{0}}} \oplus \frac{H_{\bar{1}}}{N_{\bar{1}}},$$
which implies that $(X_{\bar{0}}+ Y_{\bar{0}}) \oplus (X_{\bar{1}} + Y_{\bar{1}}) = (H_{\bar{0}} + H_{\bar{1}}).$ 
\begin{align*}
H' & = [(X_{\bar{0}}+ Y_{\bar{0}}) \oplus (X_{\bar{1}} + Y_{\bar{1}}), (X_{\bar{0}}+ Y_{\bar{0}}) \oplus (X_{\bar{1}} + Y_{\bar{1}})]\\
& = (X_{\bar{0}} +Y_{\bar{0}})' + (X_{\bar{1}} +Y_{\bar{1}})' \oplus [(X_{\bar{0}} +Y_{\bar{0}}), (X_{\bar{1}} +Y_{\bar{1}})]\\
& = (X_{\bar{0}}'+ X_{\bar{1}}' \oplus [X_{\bar{0}}, X_{\bar{1}}]) + (Y_{\bar{0}}'+ Y_{\bar{1}}' \oplus [Y_{\bar{0}}, Y_{\bar{1}}]) + [X_{\bar{0}}\oplus X_{\bar{1}}, Y_{\bar{0}} \oplus Y_{\bar{1}}]
\end{align*} 
with $ [X_{\bar{0}}\oplus X_{\bar{1}}, Y_{\bar{0}} \oplus Y_{\bar{1}}] \subset (X_{\bar{0}}\cap Y_{\bar{0}}) \oplus (X_{\bar{1}}\cap Y_{\bar{1}})$. 
Now consider the map
\begin{align*}
(X_{\bar{0}} + Y_{\bar{0}}) \oplus (X_{\bar{1}} + Y_{\bar{1}}) & \longrightarrow (H/N)_{\bar{0}} \oplus (H/N)_{1}\\
(x_{\bar{0}}, y_{\bar{0}}; x_{\bar{1}}, y_{\bar{1}}) &\mapsto (x_{\bar{0}} + y_{\bar{0}}) + N_{\bar{0}}\oplus (x_{\bar{1}}+y_{\bar{1}}) + N_{\bar{1}}. 
\end{align*}
Kernel of the above map is $N_{\bar{0}} \oplus N_{\bar{1}} = (X_{\bar{0}}\cap Y_{\bar{0}}) \oplus (X_{\bar{1}}\cap Y_{\bar{1}})$. Since $N\subset H'$ we have 
$$N = H' \cap N =( (X_{\bar{0}}'+ X_{\bar{1}}')\cap N_{\bar{0}}) \oplus ([X_{\bar{0}}, X_{\bar{1}}] \cap N_{\bar{1}})+ ((Y_{\bar{0}}'+ Y_{\bar{1}}')\cap N_{\bar{0}}) \oplus ([Y_{\bar{0}}, Y_{\bar{1}}]\cap N_{\bar{1}}) \oplus [(X_{\bar{0}} \oplus X_{\bar{1}}), (Y_{\bar{0}} \oplus Y_{\bar{1}})].$$ 

As $N \subset Z(H)$ and $N\subset (X_{\bar{0}}\cap Y_{\bar{0}}) \oplus (X_{\bar{1}}\cap Y_{\bar{1}}), N_{\bar{0}}\subset Z(X_{\bar{0}})\cap Z(Y_{\bar{0}})$ and $N_{\bar{1}}\subset  Z(X_{\bar{1}})\cap Z(Y_{\bar{1}})$, using Lemma \ref{lem9}, we have
$
\dim ((X_{\bar{0}} +X_{\bar{1}})'\cap N ) \leq \dim \mathcal{M}(A_{\bar{0}}\oplus A_{\bar{1}})$ and $\dim ((Y_{\bar{0}} + Y_{\bar{1}})'\cap N ) \leq \dim \mathcal{M}(B_{\bar{0}}\oplus B_{\bar{1}})$. Therefore, 
$$\dim \mathcal{M}(A\oplus B) = \dim N \leq \dim \mathcal{M}(A) + \dim \mathcal{M}(B) + \dim [X,Y].$$

 Now consider the following linear transformation 
 $$f:\frac{A}{A'}\times \frac{B}{B'}\longrightarrow [X, Y]$$  and defined as $f(\bar{a}, \bar{b})=[x_{a},y_{b}]$, where $x_{a} + N \longrightarrow a$ and $x_{b} + N \longrightarrow b$ comes the isomorphism $X/N \cong A$ and $Y/N \cong B$.
  It can be checked that $f$ is well defined bilinear mapping onto $[X, Y]$. Hence $f$ exists and we have $\dim([X ,Y]) \leq \dim (A/A'\otimes B/B')$. So
$$\dim \mathcal{M}(A\oplus B)  \leq \dim \mathcal{M}(A) + \dim \mathcal{M}(B) + \dim (A/A' \otimes B/B').$$

\smallskip
 Let the Lie superalgebra pairs $(U, L)$ and $(V, L)$ be maximal defining pairs for $A$ and $B$ respectively. Let $P$ be a vector superspace complement to $U^{'}$ in $U$ and $Q$ be the vector superspace complement to $V^{'}$ in $V$. Let $\{u_{i}, u_{m+i}\}$ and  $\{v_{j}, v_{n+j}\}$ be a bases for $P$ and  $Q$ respectively. Let $W=W_{\bar{0}}\oplus W_{\bar{1}}$ be a vector superspace  with $[u_{i}, v_{j}], [u_{m+i}, v_{n+j}]$ and $[u_{m+i}, v_{j}], [u_{i}, v_{n+j}]$ are the basis for $W_{\bar{0}}$ and $W_{\bar{1}}$ respectively. Then 
 \[\dim W= \dim (U/U')\dim(V/V').\]
Let $S= U \oplus V \oplus W$ with $S_{0}= U_{\bar{0}} \oplus V_{\bar{0}} \oplus W_{\bar{0}}$ and $S_{1}= U_{\bar{1}}\oplus V_{\bar{1}}\oplus W_{\bar{1}}$. Here we define a multiplication on $S$ that extends the multiplication on $U$ and $V$ such that $[U^{'}, V]=[V^{'}, U]=0$ and $W \subset Z(S)$. Hence we have $[W, U]=0=[W,V]$ and also $[W, W] \subset [W, S]=0$. From construction $[P, Q]=W$ and
\begin{align*}
[U, V] &=[U_{\bar{0}}\oplus U_{\bar{1}}, V_{\bar{0}}\oplus V_{\bar{1}}] = [(P_{\bar{0}}+U_{\bar{0}}'+U_{\bar{1}}')\oplus (P_{\bar{1}}+[U_{\bar{0}},U_{\bar{1}}]), (Q_{\bar{0}}+V_{\bar{0}}' + V_{\bar{0}}') \oplus (Q_{\bar{1}}+[V_{\bar{0}}, V_{\bar{1}})]] \\
&= [U_{\bar{0}}'+U_{\bar{1}}'+[U_{\bar{0}}, U_{\bar{1}}], V_{\bar{0}}]+[[U_{\bar{0}},U_{\bar{1}}]+U_{\bar{0}}'+U_{\bar{1}}', V_{\bar{1}}]+[P_{\bar{0}}+P_{\bar{1}}, [V_{\bar{0}}, V_{\bar{1}}]+V_{\bar{0}}'+V_{\bar{1}}^{'}] \\
& +([P_{\bar{0}},Q_{\bar{0}}]+[P_{\bar{1}}, Q_{1}])\oplus ([P_{\bar{0}}, Q_{1}]+[P_{\bar{1}}, Q_{\bar{0}}])\\
&=[U', V]+[U-U', V']+W =  W.
\end{align*}
 The last equality follows as $[U', V']\subset [U', V]=0$.
So multiplication in $S=U\oplus V \oplus W$ defined as, $[U', V]=[V', U]= [U, W]=[V, W] = 0$ and $[U,V]=W$, makes it a Lie superalgebra. Let $T= L+ M+ W$ where $T_{\bar{0}}= L_{\bar{0}}\oplus M_{\bar{0}}\oplus W_{\bar{0}}$ and $T_{\bar{1}}= L_{\bar{1}}\oplus M_{\bar{1}} \oplus W_{\bar{1}}$ . Also we have $L \subset U' \cap Z(U)$ and $M \subset V'\cap Z(V)$. Now
\begin{align*}
S' &=[(U_{\bar{0}}+ V_{\bar{0}}+ W_{\bar{0}}) \oplus (U_{\bar{1}}+V_{\bar{1}}+W_{\bar{1}}), (U_{\bar{0}}+ V_{\bar{0}}+ W_{\bar{0}})\oplus (U_{\bar{1}}+V_{\bar{1}}+W_{\bar{1}})]\\
&=U_{\bar{0}}'+V_{\bar{0}}'+[U_{\bar{0}}, V_{\bar{0}}]+[V_{\bar{0}}, U_{\bar{0}}]+U_{\bar{1}}'+[U_{\bar{1}}, V_{\bar{1}}]+[V_{\bar{1}}, U_{\bar{1}}]+V_{\bar{1}}^{'}\oplus ([U_{\bar{0}}, U_{\bar{1}}]\\ 
&+ [U_{\bar{0}}, V_{\bar{1}}]+[V_{\bar{0}}, U_{\bar{0}}]+[V_{\bar{0}}, V_{\bar{1}}]+[U_{\bar{1}}, U_{\bar{0}}]+[U_{\bar{1}}, V_{\bar{0}}]+[V_{\bar{1}}, U_{\bar{0}}]+[V_{\bar{1}}, V_{\bar{0}}])\\
&=U_{\bar{0}}'+U_{\bar{1}}'+[U_{\bar{0}}, U_{\bar{1}}]+[U_{\bar{1}}, U_{\bar{0}}]+(V_{\bar{0}}'+V_{\bar{1}}'+[V_{\bar{0}}, V_{\bar{1}}]+[V_{\bar{1}}, V_{\bar{0}}])+W\\
&= U'+V'+W.
\end{align*} 
Further, $L \subset U^{'}\cap Z(U), M \subset V^{'}\cap Z(V)$ and $S^{'}= U^{'}+V^{'}+W$ which implies $T\subset S'$ and $T\subset Z(S)$. Now consider $S/T =  (U + V + W)/(L + M + W)$ where
\[\left(S/T\right)_{\bar{0}} =(U_{\bar{0}}+V_{\bar{0}}+W_{\bar{0}})/(L_{\bar{0}}+M_{\bar{0}}+W_{\bar{0}} )= U_{\bar{0}}/L_{\bar{0}}+ V_{\bar{0}}/M_{\bar{0}} \cong(A_{\bar{0}}+B_{\bar{0}}),\]
 and 
\[\left(S/T\right)_{\bar{1}} =(U_{\bar{1}}+V_{\bar{1}}+W_{\bar{1}})/(L_{\bar{1}}+M_{\bar{1}}+W_{\bar{1}} )= U_{\bar{1}}/L_{\bar{1}}+ V_{\bar{1}}/M_{\bar{1}} \cong(A_{\bar{1}}+B_{\bar{1}}). \]
Thus, all together we have $\left(S/T\right)_{\bar{0}}\oplus \left(S/T\right)_{\bar{1}}\cong (A_{\bar{0}}+B_{\bar{0}})\oplus(A_{\bar{1}}+B_{\bar{1}})$, and hence $(S=S_{\bar{0}}\oplus S_{\bar{1}}, T=T_{\bar{0}}\oplus T_{\bar{1}})$ is a maximal defining pair for $(A_{\bar{0}}+B_{\bar{0}})\oplus(A_{\bar{1}}+B_{\bar{1}})$. 
Therefore,
 \begin{align*}
& \dim \mathcal{M}(A\oplus B) = \dim(T) =
\dim(T_{\bar{0}}\oplus T_{\bar{1}}) = \dim (L_{\bar{0}}\oplus L_{\bar{1}})+ \dim (M_{\bar{0}}\oplus M_{\bar{1}})+ \dim (W_{\bar{0}}\oplus W_{\bar{1}})\\
& = \dim \mathcal{M}(A_{\bar{0}}\oplus A_{\bar{1}})+\dim \mathcal{M}(B_{\bar{0}}\oplus B_{\bar{1}})+ \dim \left(\frac{U_{\bar{0}}}{U_{\bar{0}}^{'}+U_{\bar{1}}^{'}} \cdot\frac{V_{\bar{0}}}{V_{\bar{0}}^{'} + V_{\bar{1}}^{'}}\right)+\dim\left(\frac{U_{\bar{1}}}{[U_{\bar{0}}, U_{\bar{1}}]}\cdot \frac{V_{\bar{1}}}{[V_{\bar{0}}, V_{\bar{1}}]}\right) \\
&= \dim \mathcal{M}(A)+\dim \mathcal{M}(B)+\dim \left(\frac{A_{\bar{0}}}{A_{\bar{0}}^{'}}\cdot \frac{B_{\bar{0}}}{B_{\bar{0}}^{'}}\right)+\dim \left(\frac{A_{\bar{1}}}{A_{\bar{1}}^{'}}\cdot\frac{B_{\bar{1}}}{B_{\bar{1}}^{'}}\right)\\
&= \dim \mathcal{M}(A)+\dim \mathcal{M}(B)+ \dim (A/A'\otimes B/B').
\end{align*}
\end{proof}

\section{Special Heisenberg Lie superalgebras}\label{shl}
Heisenberg Lie superalgebras play an important role in Physics.  A Heisenberg Lie superalgebra is by definition a two step nilpotent Lie superalgebra with $1$ dimensional center. We say that a Lie superalgebra $L$ is a Heisenberg Lie superalgebra if it has a 1-dimensional homogeneous center $Z(L)$ such that $[L,L] \subseteq Z(L)$ \cite{RSS2011}. Over a algebraically closed field, all finite dimensional Heisenberg Lie superalgebra split precisely into two types, i.e., one is with even center and another is with odd center. In this paper, we focus only on special Heisenberg Lie superalgebra with even center. 
\begin{definition}\label{def13}
The Lie superalgebra $L=L_{\bar{0}}\oplus L_{\bar{1}}$ is called special Heisenberg if $[L,L]=L'=Z(L)$ and $\dim L'=1$.
\end{definition}

In the following theorem, we prove every special Heisenberg Lie superalgebras have dimension $2m+n+1$ for some non-negative integers $m$ and $n$. 

\begin{theorem}\label{th14}
Every special Heisenberg Lie superalgebra with even center have dimension $(2m+1\mid n)$, is isomorphic to $H(m,n) = H_{\bar{0}}\oplus H_{\bar{1}}$ where 
\begin{equation}
H_{\bar{0}}=<  x_{1},\ldots,x_{m},x_{m+1},\ldots,x_{2m};\, z\mid [x_{i},x_{m+i}]=z, i=1,\ldots,m]> 
\end{equation}
and
\begin{equation}
H_{\bar{1}}=<y_{1},\ldots,y_{n}\mid [y_{j},y_{j}]=z, j=1,\ldots,n>.
\end{equation}
\end{theorem}
\begin{proof}
We apply induction two times to prove the theorem. First we prove the result  for $m=0$. Let $L$ be a Heisenberg Lie superalgebra generated by $<z, y_{1},y_{2},\ldots,y_{n}>$. Here our claim is $L\cong H(0,n)$. If $n=1$, we have $L=<z, y_{1}>$. Since $L$ is Heisenberg Lie superalgebra, $L'=Z(L)$ and $\dim L'=1$. So without loss of generality we may assume that $L'=<z>$. As $z\in Z(L)$, we have $[z,y_{1}]=0$. Now if $[y_{1},y_{1}]=0$ then $y_{1}\in Z(L)$ which is a contradiction. Thus, $[y_{1}, y_{1}]=\alpha z$ for some $\alpha\neq 0$. Hence by considering the basis $\{z, \frac{1}{\alpha}y_{1}\}$, we have $L\cong H(0,1)$.

\smallskip
Assume that the result is true for $\dim L \leq n$. Now, consider  $\dim L = n+1$ and $L'=<z>$. Thus, there exist an element $y_{i}\in L$ such that $[y_{i},y_{i}]=z$. Consider the following map
\begin{equation}
\mathrm{ad}_{y_{i}}:L\rightarrow L'=<z>.
\end{equation}
Using the definition of adjoint map 
$$
\mathrm{ad}_{y_{i}}(y_{1})=\cdots=\mathrm{ad}_{y_{i}}(y_{i-1})=\mathrm{ad}_{y_{i}}(y_{i+1})=\cdots=\mathrm{ad}_{y_{i}}(y_{n})=0.
$$
 Let us denote that $N_{y_{i}}=\ker \mathrm{ad}_{y_{i}}$. It is easy to see that $\dim N_{y_{i}}=n$. For $1\leq i \leq n$, we have
 \begin{equation}
 L=N_{y_{i}}+<y_{i}>. 
 \end{equation}
Here our claim is $N_{y_{i}}$ is a special Heisenberg Lie superalgebra. 

\smallskip
Clearly $Z(L)\subset Z(N_{y_{i}})$. Conversly, if $u \in Z(N_{y_{i}})$ then $[u,r]=0$ for all $r \in N_{y_{i}}$ and also $[u,y_{i}]=0$. Thus $u \in Z(L)$. This implies $Z(L)=Z(N_{y_{i}})=<z>$. 

\begin{equation}
N_{y_{i}}' \subseteq L'=<z> = Z(N_{y_{i}})
\end{equation}
If $N_{y_{i}}'=0$ then $Z(N_{y_{i}})=N_{y_{i}}$. This is a contradiction since $\dim N_{y_{i}} = n > 1$. Hence $\dim N_{y_{i}}'=1$, which implies $N_{y_{i}}$ is a special Heisenberg Lie superalgebra with dimension $n$. Now by considering the induction hypothesis there exist   a basis $\{z,y_{1}, \cdots,y_{i-1},y_{i},\cdots y_{n}\}$  for $N_{y_{i}}$ such that $[y_{j},y_{j}]=z$ for $1\leq j\leq n$ and $j\neq i$ and also we have  assumed $[y_{i},y_{i}]=z$. Hence all together we conclude $L\cong H(0,n)$.

\smallskip
Next we induct on $m$. We have already proved that $L\cong H(0,n)$, i.e., the result is true for $m=0$. Let us assume that result is true for $\dim L \leq 2m+n-1$ where $\dim L_{\bar{0}}=2m-1$ and $\dim L_{\bar{1}}=n$. Our claim is to show that the result is true for $\dim L \leq 2m+n+1$ where $\dim L_{\bar{0}}=2m+1$ and $\dim L_{\bar{1}}=n$. Let $L$ be a Heisenberg Lie superalgebra generated by $<z, x_{1},\cdots, x_{m},x_{m+1},\cdots, x_{2m}, y_{1}, \cdots, y_{n} >$ with the assumption that $[y_{j},y_{j}]=z$ for $1\leq j \leq n$. Without loss of generality we can assume $L'=<z>=Z(L)$. Thus there exists some $x_{i}$ such that $[x_{i},x_{m+i}]=z$. Consider the maps
\begin{equation}
\mathrm{ad}_{x_{i}} : L\rightarrow L'=<z>\; \mathrm{and}\; \mathrm{ad}_{x_{m+i}} : L\rightarrow L'.
\end{equation}
Clearly, one can see that
$$N_{x_{i}} := \ker \mathrm{ad}_{x_{i}}=\{z, x_{k}, y_{j}: 1\leq k\leq 2m, 1\leq j \leq n, k \neq m+i \}$$
and 
$$N_{x_{m+i}} := \ker \mathrm{ad}_{x_{m+i}}=\{z, x_{k}, y_{j}: 1\leq k \leq 2m, 1\leq j \leq n , k \neq i \}.$$
Therefore we have 
$$
 L=N_{x_{i}}\cap N_{x_{m+i}}+<x_{i},x_{m+i}>.$$ 
 Also note that $\dim N_{x_{i}}\cap N_{x_{m+i}}=2m+n-1$. If we can show that $N_{x_{i}}\cap N_{x_{m+i}}$ is a special Heisenberg Lie superalgebra then induction hypothesis will do rest of the work. 

\smallskip 
It is easy to see that $Z(L)\subseteq Z(N_{x_{i}}\cap N_{x_{m+i}})$. If $u \in N_{x_{i}}\cap N_{x_{m+i}}$ then $[u,m] = 0$ for all $m \in N_{x_{i}}\cap N_{x_{m+i}}$, i.e., $[u,m] = 0$ for $m \neq x_{i}, x_{m+i}$. On the other hand $[u,x_{i}]=0$ and $[u,x_{m+i}]=0$. Thus, $u \in Z(L)$. Hence 
$Z(L)=<z>=Z(N_{x_{i}}\cap N_{x_{m+i}})$. Therefore,
$$
(N_{x_{i}}\cap N_{x_{m+i}})'\subseteq Z(L)=<z>=Z(N_{x_{i}}\cap N_{x_{m+i}}).$$
If $(N_{x_{i}}\cap N_{x_{m+i}})' = 0$, then $Z(N_{x_{i}}\cap N_{x_{m+i}}) = N_{x_{i}}\cap N_{x_{m+i}}$ but this is impossible as $\dim (N_{x_{i}}\cap N_{x_{m+i}}) =2m+n-1 > 1 $. Thus, $N_{x_{i}}\cap N_{x_{m+i}}$ is a special Heisenberg Lie superalgebra. Using induction hypthesis, there exist a basis $\{z,x_{1},\cdots, x_{i-1},x_{i+1},\cdots, x_{m+i-1},x_{m+i+1},\cdots, x_{2m}, y_{1}, \cdots, y_{n} \}$ such that $[x_{k},x_{m+k}]=z$ for $1\leq k\neq i\leq 2m$. By adding the elements ${x_{i},x_{m+i}}$ to the above basis we have $\dim L= 2m+n+1$ and the result follows from the fact that $[x_{i},x_{m+i}]=z$. This completes the proof of the theorem.
\end{proof}

In the following theorem, we compute multiplier of the special Heisenberg Lie superalgebra.
\begin{theorem}\label{th15}
Let $H(m,n)$ be a special Heisenberg Lie superalgebra with even center of dimension $(2m+1\mid n)$. Then 
$$
\dim \mathcal{M}(H(m,n))=\begin{cases}
2m^2-m-1+2mn+n(n+1)/2 & \mbox{if}\ m + n \geq 2 \\
2 & \mbox{if }\ m=0, n=1\\

2  &\mbox{if }\ m=1, n=0.
\end{cases}
$$
\end{theorem}

\begin{proof}
Let us first assume that $m + n  \geq 2$. We are interested to find a stem cover for the special Heisenberg Lie superalgebra $H(m,n)$. Let $W$ be a vector superspace with a basis 

$$\{w_{i},v_{j},w_{k,l}, v_{k',l'}, \eta_{k}, \eta_{j}',\gamma_{k,j}\}.$$ Consider a vector superspace $C$ with a basis, 

$$\{z, x_{1}, x_{2}, \cdots , x_{m}, x_{m+1}, \cdots, x_{2m},\ y_{1}, y_{2}, \cdots, y_{n}\}$$
 and put $K= C+ W$. We want to define a superbracket in $K$ which turn $K$ into a Lie superalgebra.
\begin{align*}
[x_{i}, x_{m+i}] & = z + w_{i},\;\; 1\leq i \leq m\\
[y_{j}, y_{j}]   & = z + v_{j},\;\; 1\leq j \leq n\\
[x_{k}, x_{l}]   & = w_{k, l},\; \mathrm{for}\; 1 \leq k < l \leq 2m\\
[y_{k'}, y_{l'}]&= v_{k',l'},\; \mathrm{for}\; 1 \leq k' < l'\leq n
\end{align*}
such that $(k, l) \not \in \{(i_{1}, m+{i_{1}}) | 1 \leq i_{1} \leq m\}$ and $(k', l') \not \in \{(i_{2}, i_{2})| 1 \leq i_{2} \leq n \}$.

Also, 
\begin{align*}
[x_{k}, z] &= \eta_{k}\; \mathrm{for}\;1 \leq k \leq 2m\\
[y_{j}, z] &= \eta_{j'}\; \mathrm{for}\;1 \leq j \leq n\\
[x_{k}, y_{j}] &= \gamma_{k,j}\; \mathrm{for}\;1 \leq k \leq 2m\; \mathrm{and}\;1 \leq j \leq n.\\
\end{align*}

Now we see that the above multiplication is true if we replace $z$ by $\zeta,\; w_{1} = 0,\;w_{i}$ by $\hat{w}_{i} = w_{i} - w_{1}$ and $v_{j}$ by $\hat{v}_{j} = v_{j} - w_{1}$. Let $\zeta=[x_{1}, x_{m+1}] = z + w_{1}$. Then for $2 \leq i \leq m$, and $1 \leq j \leq n$ we have 
\begin{align*}
 [x_{i}, x_{m+i}] &= \zeta + (w_{i}-w_{1}) = \zeta + \hat{w}_{i},\\
  [y_{j}, y_{j}]  &= \zeta +(v_{j}-w_{1}) = \zeta + \hat{v}_{j}.
 \end{align*}
Now,
\begin{align*}
[x_{k}, \zeta] = [x_{k}, z+w_{1}] &= [x_{k}, z] = \eta_{k}\\
[y_{j}, \zeta] = [y_{j}, z+w_{1}] &= [y_{j}, z] = \eta_{k}.'
\end{align*}

For $1\leq k \leq 2m, k \neq m+1$ and using the super Jacobi identity for the Lie superalgebra 
\begin{align*}
0 & =J(x_{1}, x_{m+1}, x_{k})\\
  & =[[x_{1},x_{m+1}], x_{k}]+[[x_{m+1}, x_{k}],x_{1}]+[[x_{k}, x_{1}],x_{m+1}]\\
  & = [\zeta , x_{k}] + [w_{m+1,k}, x_{1}] + [w_{k,1}, x_{m+1}]  \\
  & = [\zeta, x_{k}] + 0 + 0 = -\eta_{k}, 
\end{align*}
Again taking the Jacobi idenetity of $(x_2, x_{m + 2}, x_{m + 1})$ one can see that $n_{m+1} = 0$. Similarly, we can make $\eta_{j}'$ equal to zero using the super Jacobi identity for the triple $(x_{1},x_{m+1},y_{j})$ for $1\leq j \leq n$, i.e., $0=J(x_{1},x_{m+1},y_{j})=\eta_{j}'$. Therefore $K$ is a Lie superalgebra and maximal basis of $W$ consists of all elements 

\begin{align*}
w_{i} &\quad 1< i\leq m,\\
v_{j} & \quad 1\leq j \leq n,\\
w_{k,l} & \quad 1 \leq k < l \leq 2m,\; \mathrm{and}\;\; (k, l) \not \in \{(i_{1}, m+{i_{1}}) \mid 1 \leq i_{1} \leq m\} \\
v_{k',l'} & \quad 1 \leq k' < l' \leq n,\; \mathrm{and}\;\; (k', l') \not \in \{(i_{2}, i_{2}) \mid 1 \leq i_{2} \leq n\}\\
\gamma_{k,j} & \quad 1 \leq k \leq 2m\;\mathrm{and}\;1 \leq j \leq n.
 \end{align*}
 
One can observe that $w_{i}, v_{j}, w_{k,l}, v_{k',l'}\in W_{\bar{0}}$ and $\gamma_{k,j} \in W_{\bar{1}}$. Now, $0\longrightarrow W \longrightarrow K \longrightarrow H(m,n) \longrightarrow 0$ is a stem cover for $H(m,n)$. Thus,
 
\begin{align*}
  \dim W = \dim \mathcal{M}(H(m,n)) &= m-1+n+\binom{2m}{2}-m+\binom{n}{2}+2mn \\
  & = 2m^2-m-1+2mn+n(n+1)/2.
 \end{align*}
Now, suppose that $m=0$ and $n=1$. Let $W$ and $C$ be the vector superspace with bases $\{w, \eta\}$ and $\{z;y\}$ respectively. We show, $K=W+C$ is a Lie superalgebra. Writting the super brackets as
$$
[y,y]=z+w;\; \mathrm{and}\; [y,z] = \eta,
$$
 K is a Lie superalgebra. The maximal basis of $W$ is $\{w, \eta\}$. Therefore, $0\longrightarrow W \longrightarrow K \longrightarrow H(m,n) \longrightarrow 0$ is a stem cover for $H(m,n)$ and thus 
 $$
 \dim W = \dim \mathcal{M}(H(m,n))=2.
 $$
 Finally when $m=1$ and $n=0$, $H(m,n)$ is a Heisenberg Lie algebra of dimension $3$ and in this case dimension of the multiplier is $2$ as shown in \cite{Batten1993}. This completes the proof of the theorem.
\end{proof}
The following is a consequence of Theorem \ref{th15}.
\begin{corollary}[See Example 3 in \cite{BS1996}]
Let $H(m)$ be a special Heisenberg Lie algebra of dimension $2m+1$. Then  
$$
\dim \mathcal{M}(H(m))=\begin{cases}
2m^2-m-1 & \mbox{if}\ m > 1 \\
2 & \mbox{if }\ m= 1.
\end{cases}
$$
\end{corollary} 
The following is an immediate consequence of Theorem \ref{lem8} and Theorem \ref{th12} and \ref{th15}.
\begin{corollary}
Suppose $A$ is an abelian Lie superalgebra of dimension $(m-2p-1\mid n-q)$ and suppose $X:= H(p,q) \oplus A$ is a Lie superalgebra. Then 
$$
\dim \mathcal{M}(X)=\begin{cases}
\frac{1}{2}\left[(m+n)^2-(3m + n)\right] & \mbox{if}\ p+q > 1 \\
 2 + \frac{1}{2}\left[(m+n)^2-(3m + n)\right] & \mbox{if }\ p+q= 1.
\end{cases} 
$$
\end{corollary}

\section{Main Theorem}\label{mt}

The following result provides an upper bound for dimension of multiplier of nilpotent Lie superalgebra which is less than the bound in Corollary \ref{cor11b} except for the cases $L \cong H(0,1)$ and $L\cong H(1,0).$

\begin{theorem}\label{th16}
Let $L = L_{\bar{0}} \oplus L_{\bar{1}}$ be a nilpotent Lie superalgebra of dimension $(m\mid n)$ and $\dim L'= (r\mid s)$ with $r+s \geq 1$. Then 
\begin{equation}
\dim \mathcal{M}(L)\leq \frac{1}{2}\left[(m + n + r + s - 2)(m + n - r -s -1) \right] + n + 1.
\end{equation}
Moreover, if $r+s = 1$, then the equality holds if and only if 
$$
L \cong \begin{cases}
H(1,0) \oplus A_{1} & \mbox{if}\ r = 1, s = 0; \\
H(0,1) \oplus A_{2} & \mbox{if}\ r = 0, s = 1,
\end{cases}
$$
where $A_{1}$, $A_{2}$ are abelian Lie superalgebras with $\dim A_{1}=(m-3 \mid n)$, $\dim A_{2}=(m-1 \mid n-1)$ respectively and $H(1,0), H(0,1)$ are special Heisenberg Lie superalgebras of dimension $3$ and $2$ respectively.
\end{theorem}

\begin{proof}
First assume that $\dim L'= (1\mid 0)$. Then $L/L' = L_{\bar{0}}/(L'_{\bar{0}}+L'_{\bar{1}}) \oplus L_{\bar{1}}/[L_{\bar{0}}, L_{\bar{1}}]$ is an abelian superalgebra of $\dim (L/L') =(m-1 \mid n) $. Since $L$ is nilpotent we have $( L'_{\bar{0}}+ L'_{\bar{1}}) \subseteq Z(L)_{\bar{0}}$ and $[L'_{\bar{0}},L'_{\bar{1}}] \subseteq Z(L)_{\bar{1}}$. Let us suppose that $(H/L')_{\bar{0}} \oplus  (H/L')_{\bar{1}}$ is a complement of $(Z(L)/L')_{0} \oplus (Z(L)/L')_{1}$ in $(L/L')_{\bar{0}} \oplus (L/L')_{\bar{1}}$. Hence, $L = (H_{\bar{0}} + Z(L)_{\bar{0}} \oplus (H_{\bar{1}} + Z(L)_{\bar{1}}$ and also it is easy to see $L'=H'$, i.e., $L'_{\bar{0}} + L'_{\bar{1}} = H'_{\bar{0}} + H'_{\bar{1}}$ and $[L_{\bar{0}}, L_{\bar{1}}] = [H_{\bar{0}}, H_{\bar{1}}]$. Now our claim is $(H \cap Z(L))_{\bar{0}} \subseteq H_{\bar{0}}$ and $(H \cap Z(L))_{\bar{1}} \subseteq H_{\bar{1}}$.

\smallskip 
Since $H/L'$ and $Z(L)/L'$ are complementary subspaces we have
\begin{equation*}
(H_{\bar{0}} \cap Z(L)_{\bar{0}})/ (L'_{\bar{0}}+L'_{\bar{1}}) \oplus (H_{\bar{1}} \cap Z(L)_{\bar{1}})/[L'_{\bar{0}}, L'_{\bar{1}}] = \phi
\end{equation*} 
which further implies that
\begin{equation*}
(H_{\bar{0}} \cap Z(L)_{\bar{0}})/ (L'_{\bar{0}}+L'_{\bar{1}})  = \phi\;\; \mathrm{and}\;\; (H_{\bar{1}} \cap Z(L)_{\bar{1}})/ [L'_{\bar{0}}, L'_{\bar{1}}] = \phi. 
\end{equation*}
 The first part of the above equation gives $(H_{\bar{0}} \cap Z(L)_{\bar{0}}) \in (L'_{\bar{0}}+L'_{\bar{1}}) = (H'_{\bar{0}}+ H'_{\bar{1}})\subseteq H_{\bar{0}}$ and similarly the second part gives $(H_{\bar{1}} \cap Z(L)_{\bar{1}}) \in [H'_{\bar{0}}, H'_{\bar{1}}] \subseteq H_{\bar{1}}$ and this completes the proof of the claim. Hence $H\cap Z(L)\subseteq H$ and so $Z(H) = L'$.

\smallskip
Since $(L'_{\bar{0}} + L'_{\bar{1}}) \subseteq Z(L)_{\bar{0}}$ and  $[L_{\bar{0}}, L_{\bar{1}}] \subseteq Z(L)_{\bar{1}}$, $L'$ must have a complement $K = K_{\bar{0}} \oplus K_{\bar{1}}$ in $Z(L)_{\bar{0}} \oplus Z(L)_{\bar{1}}$. Hence 
$(L'_{\bar{0}} + L'_{\bar{1}}+ K_{\bar{0}}) \oplus ([L_{\bar{0}}, L_{\bar{1}}]+ K_{\bar{1}}) = Z(L)_{\bar{0}} \oplus Z(L)_{\bar{1}}$ which implies that $L\cong (H_{\bar{0}} \oplus K_{\bar{0}}) \oplus (H_{\bar{1}} \oplus K_{\bar{1}})$. By Theorem \ref{th12},
\begin{equation}\label{eq16a}
\dim \mathcal{M}(L) = \dim \mathcal{M}(H) +\dim \mathcal{M}(K) +\dim \left(H/H'\otimes K/K'\right).
\end{equation}

Here $H = H_{\bar{0}}\oplus H_{\bar{1}}$ is a special Heisenberg Lie superalgebra of dimension $(2p+1\mid q)$ for some non-negative integers $p,q$ and $K = K_{\bar{0}}\oplus K_{\bar{1}}$ is an abelian Lie superalgebra.

\smallskip
{\bf Case 1} $(p + q \geq 2).$

\smallskip
 From Theorem \ref{lem8} and Theorem \ref{th15},
\begin{align*}
\dim \mathcal{M}(H) & = 2p^2 - p - 1 + 2pq + q(q+1)/2,\\
\dim \mathcal{M}(K) & = \frac{1}{2}\left[ (m -2p -1 + n-q)^2 + n-q -m +2p +1)\right],\\
\dim (K/K' \otimes H/H') & = (m-2p-1 + n-q)(2p +q),  
\end{align*}

and putting these in \eqref{eq16a},
\begin{align*}
\dim (\mathcal{M}(L))  = & 2p^2 - p - 1 + 2pq + q(q+1)/2\\
					     & +\frac{1}{2}\left[ (m -2p -1 + n-q)^2 + n-q -m +2p +1)\right]\\
					     & + (m-2p-1 + n-q)(2p +q)\\
					     & = \frac{1}{2}\left[ (m+n)^2- (3m+n)\right] < \frac{1}{2}\left[ (m+n-1)(m+n-2)\right] + n +1.   
\end{align*}

\smallskip
{\bf Case 2} $(p + q = 1).$ 

\smallskip
Here we have two choice for $p$ and $q$. 

\smallskip
{\bf Subcase 1} ($p=1$ and  $q = 0).$ 

\smallskip
Thus, $\dim H = (3\mid 0)$ and $\dim K = (m-3\mid n)$. Using Theorem \ref{lem8} and Theorem \ref{th15} in \eqref{eq16a}, 
\begin{align*}
\dim (\mathcal{M}(L))  = & 2 + \frac{1}{2}\left[ (m-3+ n)^2 + (n-m+3)\right] + 2(m-3 + n)\\
					     & = 2 + \frac{1}{2}\left[ (m+n)^2- (3m+n)\right].   
\end{align*}

\smallskip
{\bf Subcase 2} $(p=0$ and  $q = 1).$

\smallskip
Here we have $\dim H = (1\mid 1)$ and $\dim K = (m-1\mid n-1)$. Similarly,
\begin{align*}
\dim (\mathcal{M}(L))  = & 2 + \frac{1}{2}\left[ (m-1+ n-1)^2 + (n-1-m+1)\right] + (m-1 + n-1)\\
					     & = 2 + \frac{1}{2}\left[ (m+n)^2- (3m+n)\right].   
\end{align*}

To complete the proof, we use the induction on $\dim L'$. Let $\dim L'= r+s >1$. Since $L'_{\bar{0}} + L'_{\bar{1}} \subseteq Z(L)_{\bar{0}}$ and $[L_{\bar{0}}, L_{\bar{1}}]\subseteq  Z(L)_{\bar{1}}$ we have $\left((L'_{\bar{0}} + L'_{\bar{1}}) \cap Z(L)_{\bar{0}}\right) \oplus \left([L_{\bar{0}},L_{\bar{1}}] \cap Z(L)_{\bar{1}}\right) \neq \phi$. Thus, one can choose an one dimensional graded ideal $K=K_{\bar{0}} \oplus K_{\bar{1}}$ of $L$ in $(L' \cap Z(L))_{\bar{0}} \oplus (L' \cap Z(L))_{\bar{1}}$. Since $\dim K =1$, again we have to consider two cases separately, i.e., $\dim K_{\bar{0}} =1$ and $\dim K_{\bar{0}} =0$. 

\smallskip
{\bf Case 1}  $(\dim K_{\bar{0}}= 1$ and $\dim K_{\bar{1}} = 0).$

\smallskip
 Now by using Theorem \ref{th11},
$$
\dim (\mathcal{M}(L)) \leq \dim (\mathcal{M}(L/K)) + \dim (\mathcal{M}(K)) + \dim (L/L'\otimes K)-1. $$
Here we use the induction hypothesis for the Lie superalgebra $L/K$ with $\dim (L/K) = (m-1\mid n )$ and $\dim (L/K)' = (r-1\mid s )$. Therefore, 
\begin{align*}
\dim (\mathcal{M}(L)) & \leq \frac{1}{2}\left[ (m+n +r+s-4)(m+n-r-s-1)\right] + n + 1 +( m+n-r-s) - 1\\
					  & =\frac{1}{2}\left[ (m+n +r+s-1)(m+n-r-s-2)\right]  + n +1.
\end{align*}

\smallskip
{\bf Case 2} $(\dim K_{\bar{0}} =0$ and $\dim K_{\bar{1}} =1).$

\smallskip
 Similarly, using the induction hypothesis 
for the Lie superalgebra $L/K$ with $\dim (L/K) = (m\mid n-1 )$ and $\dim (L/K)' = (r\mid s -1 ),$ we have 
\begin{align*}
\dim (\mathcal{M}(L)) & \leq \frac{1}{2}\left[ (m+n +r+s-4)(m+n-r-s-1)\right] + (n-1) + 1 + 1 +( m+n-r-s) - 1\\
					  & =\frac{1}{2}\left[ (m+n +r+s-1)(m+n-r-s-2)\right]  + n +1.
 \end{align*}
This completes the proof of the theorem.
\end{proof}

The following is a consequence of Theorem \ref{th16}.
\begin{corollary}[See Theorem 3.1 in \cite{NR2011}]
Let $L$ be a nilpotent Lie algebra of $\dim (L) = n$ and $\dim(L')= r\geq 1$. Then 
$$
\dim \mathcal{M}(L)\leq \frac{1}{2}(n+r-2)(n-r-1)+1. $$
\end{corollary}
Moreover, if $r=1$, then the equality holds if and only if $L\cong H(1) + A$, where $A$ is an abelian Lie algebra of $\dim (A) = n-3$.


\medskip

\end{document}